\documentclass[a4paper,11pt]{article}

\usepackage[numbers]{natbib}

\usepackage[
left=2cm,      
right=2cm,     
top=2cm,         
bottom=2cm       
]{geometry}

\usepackage{graphicx}%
\usepackage{multirow}%
\usepackage{amsmath,amssymb,amsfonts}%
\usepackage{amsthm}%
\usepackage{mathrsfs}%
\usepackage[title]{appendix}%
\usepackage{xcolor}%
\usepackage{textcomp}%
\usepackage{manyfoot}%
\usepackage{booktabs}%
\usepackage{algorithm}%
\usepackage{algorithmicx}%
\usepackage{algpseudocode}%
\usepackage{listings}%
\usepackage{arydshln} 
\usepackage{ulem}
\newtheorem{theorem}{Theorem}
\newtheorem{lemma}[theorem]{Lemma}%
\newtheorem{remark}{Remark}%

\DeclareMathOperator*{\argmin}{argmin}
\renewcommand{\Re}{\mathbb{R}}
\newcommand{\ra}[1]{\renewcommand{\arraystretch}{#1}}
\definecolor{csf}{gray}{0.0}
\definecolor{gm}{gray}{0.5}
\definecolor{wm}{gray}{0.8}

\def\bc{\mathbf{c}}

\def\bg{\mathbf{g}}

\def\bm{\mathbf{m}}

\def\bu{\mathbf{u}}
\def\bv{\mathbf{v}}
\def\bw{\mathbf{w}}
\def\bx{\mathbf{x}}

\def\bz{\mathbf{z}}

\def\nn{\textrm{n}}

\def\blambda{\boldsymbol{\lambda}}

\def\bxi{\boldsymbol{\xi}}
\def\bzeta{\boldsymbol{\zeta}}

\def\bN{\mathbb{N}}

\usepackage{authblk}    
\usepackage{orcidlink}  
\usepackage{hyperref}   
\graphicspath{{./figures/},{./}}

\title{A three-step framework for noisy image segmentation in brain MRI}

\author[1]{Laura Antonelli\orcidlink{0000-0002-4031-099X}\thanks{These authors contributed equally to this work.}}

\author[1,2]{Valentina De Simone\orcidlink{0000-0002-3357-5252}$^*$}

\author[3]{Marco Viola\orcidlink{0000-0002-2140-8094}$^*$\thanks{Corresponding author.  Email: marco.viole@dcu.ie}}

\affil[1]{Institute for High-Performance Computing and Networking, CNR,
	
	Via Pietro Castellino, 111, 80131 Naples, Italy}

\affil[2]{Department of Mathematics and Physics, University of Campania ``L. Vanvitelli'',
	
	viale Abramo Lincoln, 5, 81100 Caserta, Italy}

\affil[3]{School of Mathematical Sciences, Dublin City University,
	
	 Collins Avenue Extension, D09 V209 Dublin, Ireland}

\date{}

\begin{document}

\maketitle

\begin{abstract}
Magnetic Resonance Imaging (MRI) is essential for noninvasive generation of high-quality images of human tissues. Accurate segmentation of MRI data is critical for medical applications like brain anatomy analysis and disease detection. However, challenges such as intensity inhomogeneity, noise, and artifacts complicate this process.
To address these issues, we propose a three-step framework exploiting the idea of Cartoon-Texture evolution to produce a denoised and debiased MR image. The first step involves identifying statistical information about the nature of the noise using a suitable image decomposition. In the second step, a multiplicative intrinsic component model is applied to a smother version of the image, simultaneously reconstructing the bias and removing noise using noise information from the previous step. At the final step, standard clustering techniques are used to create an accurate segmentation.
Additionally, we present a convergence analysis of the ADMM scheme for solving the nonlinear optimization problem with multiaffine constraints resulting from the second step. Numerical tests demonstrate the effectiveness of our framework, especially in noisy brain segmentation, both from a qualitative and a quantitative viewpoint, compared to similar methods. 


\end{abstract}

\section{Introduction}\label{sec:intro}

Magnetic Resonance Imaging (MRI) is a critical and widely used medical modality due to its noninvasive nature and ability to produce high-quality images of human organs and tissues. Segmentation in MRI refers to the process of dividing the acquired image data into specific tissues or regions of interest (ROIs). In the context of brain slice MRI data, this often involves distinguishing between cerebrospinal fluid (CSF), gray matter (GM), and white matter (WM), in 2D image slices as depicted in Fig. \ref{fig:brainGT-CTD}. Research topics on brain anatomy and functionality (e.g., Alzheimer's disease, tumor detection, neurodegenerative processes, etc.) demand reliable segmentation tools since their outcomes directly impact the subsequent analysis of the brain slice MRI. Therefore, the development of any accurate MRI segmentation framework has to deal with the following issues:
\begin{itemize}
\item \textit{bias field level}: intensity inhomogeneity across the image hamper distinguishing the adjacent brain tissues;
\item \textit{noise}: Rician noise corrupts brain MRI slices due to electron thermal agitation in imaging machines, but other noise types, like random and structured noise, can be present;
\item \textit{visual artifacts}: other artifacts undermine the scans' light intensity due to body motion, blood flow, and tissues' susceptibility to magnetic fields.
\end{itemize}

\begin{figure}[ht]
\centering
\begin{tabular}{cccccc}
\includegraphics[width=.12\columnwidth]{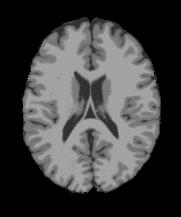} &
\includegraphics[width=.12\columnwidth]{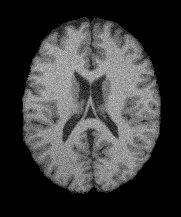} &
\includegraphics[width=.12\columnwidth]{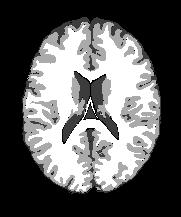} &
\includegraphics[width=.12\columnwidth]{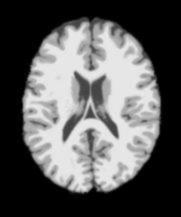} &
\includegraphics[width=.12\columnwidth]{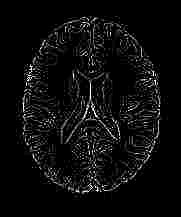}\\
\footnotesize \texttt{brain MRI slice} & \footnotesize \texttt{artifacts} & \footnotesize \texttt{ground truth} & \footnotesize \texttt{cartoon} & \footnotesize \texttt{texture} \\
\end{tabular}
\caption{The original brain MRI coming from \textsc{BrainWeb} database (slice n. 91), followed by (from left to right) its corrupted version with noise and bias, the ground truth of the brain selected  ROIs: cerebrospinal fluid in \textit{dark grey},  gray matter in \textit{grey},  and white matter in \textit{white}. The last two images depict cartoon and texture components of the MRI scan, respectively.}\label{fig:brainGT-CTD}
\end{figure}
\noindent Traditional segmentation techniques~\cite{bib:ANUF2022}, like the K-means clustering algorithm, often struggle when dealing with noise and artifacts described above in image intensity~\cite{zheng2018image}. To effectively apply these techniques, it is necessary to perform bias field correction either as a separate preprocessing step or as part of the MRI segmentation process~\cite{Li2020AMM}.
In addition, preprocessing filters or regularization techniques are applied to handle the noise susceptibility of the MRI segmentation methods, aiming to remove high-frequency components~\cite{PANKAJ2021102737}.  
The first ones often rely on the grey-level histogram, while others incorporate spatial image information to handle noise and artifacts. 
On the upside, continued advances in medical imaging technologies result in new application-specific challenges for segmentation, and standard and new methods are continuously explored, modified, and introduced~\cite{new2,new1}.
In this connection, starting from a previous related work from the same team presented in~\cite{bib:LNCS2023}, we propose SegMIC2T (Segmentation using Multiplicative Component Optimization and Cartoon-Texture evolution), a three-step framework for image segmentation in brain MRI corrupted by noise and artifacts. In detail:
\begin{itemize}
    \item[\textbf{Step 1:}] Let $I$ be the MRI slice to be segmented, the first step computes $\bar{\bv}$, a matrix of the same size as the original image containing approximate information of the noise distribution on $I$; this could be done, e.g., by applying non-linear low-pass and high-pass filters.
This step results in a decomposition of the original image $I$ in two parts, i.e., 
\begin{equation}\label{eq:ctd}
    I = \bar{I} + \bar{\bv},
\end{equation}
\noindent where $\bar{I}$ can be interpreted as a smoother version of $I$. 
With a little abuse of notation, we indicate the equation \eqref{eq:ctd} as a cartoon-texture (CT) decomposition (even though we do not aim to compute a proper CT decomposition), where $\bar{I}$ is the `cartoon' depicting the structural component, and $\bar{\bv}$ the `texture' part including oscillatory components, like texture and noise, as shown in Figure~\ref{fig:brainGT-CTD}.

\item[\textbf{Step 2:}] In the second step, the CT evolution strategy~\cite{bib:ctetris}
is applied to the multiplicative intrinsic component optimization model~\cite{MICO} (MICO), enabling the bias field reconstruction and the cartoon denoising simultaneously. Our strategy modifies the MICO model, allowing further oscillating component extraction from the smoother MRI slice by using statistical information on the noise localization. 

\item[\textbf{Step 3:}] In the third step, any standard segmentation technique can divide the tissue regions quickly, starting from the bias-corrected and denoised MRI slice. 
\end{itemize}

From the mathematical viewpoint, the CT evolution changes MICO model in a nonlinear constrained optimization problem subject to multiaffine constraints. Following \cite{bib:ADMMmultiaffine}, we introduce an Alternating Direction Method of Multipliers scheme (ADMM), proving some convergence results. 
Finally, several numerical tests exhibiting the performance of the proposed framework paired with standard K-means clustering in brain MRI segmentation are presented.

The rest of the paper is outlined as follows. In section~\ref{sec:MIC2T},  we introduce the new optimization model for MRI correction from bias and noise based on the CT evolution strategy in the MICO model. Then, an ad-hoc ADMM scheme is developed for the 
 minimization of the proposed problem. In section~\ref{sec:proof}, the convergence analysis of the ADMM scheme is detailed, and in section~\ref{sec:test}, numerical implementation and experiments are discussed. We draw some conclusions in section~\ref{sec:conc}.

 
\section{A CT evolution strategy for MRI correction}\label{sec:MIC2T}
In this section we introduce the mathematical model on which \textbf{Step 2} of the proposed strategy is based. For the sake of simplicity, we adopt the \emph{first discretize then optimize} \cite{bib:ANUF2022} approach, where the image domain consists of a grid of $ P = n_x \times n_y$ pixels,
\( \Omega = \{ (i_x,i_y) : \, i_x=1,...,n_x, \; i_y=1,...,n_y\} \), and consequently an image $I$ is a matrix of pixels with values in $[0,\infty)$.
We also assume the following multiplicative model for an MRI slice \cite{bib:VariationalMS}:
\begin{equation}\label{eq:MultIntrComp}
    I(\bxi) = b(\bxi)J(\bxi) + \nn(\bxi), \, \textrm{with} \,  \bxi \equiv (i_x,i_y) \in\Omega,
\end{equation}
where $J(\bxi)$ represents the original image, $b(\bxi)$ the bias field, and $n(\bxi)$ represents noise. Assuming as in \cite{MICO} that the brain is made of $N$ different regions of interest (i.e., tissues, in general, $N=3$), characterized by a uniform color $c_i$, one can represent the original image as a piecewise constant function, i.e., $ J(\bxi) = \sum_{i=1}^N c_i\cdot u_i(\bxi)$ with $u_i(\bxi):\Omega\rightarrow\{0,1\}$ {and $\sum_j u_j(\bxi) =1$}. Since the bias field is considered to be a smooth function, it can be represented using linear combinations of smooth functions from a specified basis of size $M$, i.e, $\{g^1(\bxi),\,\ldots,\,g^M(\bxi)\}$ (consisting, e.g., of polynomials or splines) with coefficients $\bw\in\Re^M$, i.e., $b(\bxi) = \sum_{j=1}^M w_j\cdot g^j(\bxi)$. Therefore, in this discretized setting, the bias field can be represented as a linear combination of $M$ vectors $\bg^1, \ldots, \bg^M\in\Re^P$, comprising the evaluations of the
$M$ smooth basis functions at each pixel.
Typical values for $M$ are $10$ or $20$. This results in a pixel-wise expression of the bias field as $\bw^TG_p\equiv\sum_j w_j g^j_p$. For each pixel, we also define the vector $\bu_p\in\Re^N$, representing the pixel's membership in one of the $N$ brain regions. Here, we will employ \emph{fuzzy} membership functions, meaning that for each pixel $\bu_p$ has elements in $[0,1]$ which sum up to $1$ rather than being a vector in $\{0,1\}^N$ (i.e., a vector with a single nonzero entry equal to 1).\\
\noindent To apply the CT-evolution strategy in MICO energy formulation \cite[eq. (6)]{MICO}, we consider a decomposition of the MRI slice $I$ as defined in \eqref{eq:ctd}, with $\bar{I}$ containing the piecewise-constant regions with contours and $\bar{\bv}$ the oscillating components. We remark, that the latter has the role of keeping approximate information on where noise is located in the original image.
\noindent By defining the energy potential
\begin{equation}\label{eq:MICO_potential}
    F(\bu,\bc,\bw) = \sum_{p=1}^{P} \frac{1}{2} \left( I_p - \sum_{i=1}^N \bw^\top G_p c_i(\bu_p)_i \right)^2,    
\end{equation}
reflecting the distance between the original image and the multiplicative intrinsic components model,
we can introduce our modified MICO model as follows:
\begin{equation}\label{eq:MIC2T}
    \begin{array}{rcl}
         \min\limits_{\bu,\bc,\bw,\bv}          &\;& \displaystyle F(\bu,\bc,\bw) + \mu \mathcal{D}(\bv;\bar{\bv}) \\
         \mathrm{s.t.} & & \displaystyle 0 \leq (\bu_p)_i \leq 1,\quad \forall i=1,\ldots,N,\; \forall p = 1,\ldots,P\\
                       & & \displaystyle \sum_{i=1}^N (\bu_p)_i = 1,\quad \forall p = 1,\ldots,P,\\
                       & & \displaystyle \bw^\top G_p \sum_{i=1}^N c_i(\bu_p)_i + v_p = \bar{I}_p,\quad \forall p = 1,\ldots,P
    \end{array}
\end{equation}
where $\mu$ is a positive parameter, and $\mathcal{D}$ penalises the distance between $\bv$ and $\bar{\bv}$. Our model performs the reconstruction of the multiplicative components of the structural part, $\bar{I}$, and, at the same time, attempts to extract from it the oscillatory part remnant from $\bv$. 
Since noise and general oscillatory components can be modeled as probability distributions, one could think to force the similarity between $\bv$ and $\bar{\bv}$ through the Kullback-Leibler (KL) divergence, i.e., choosing
$$\mathcal{D}(\bv;\bar{\bv}) = D_{KL}(\bv;\bar{\bv}) := \sum_{p} v_{p} \log \left( \frac{v_{p}}{\bar{v}_{p}}\right) + \bar{v}_{p} - v_{p}.$$
Although the KL divergence is not a distance, it penalizes $\bv$ for having large values where $\bar{\bv}$ is small; in this way, $\bv$ is forced to have similar support of $\bar{\bv}$ (or, in other words, pushing the method to ``extract'' noise only in those parts of the image where it was originally detected). Observe that in this case, since $\bar{\bv}$ might have negative entries, one should slightly modify the definition of the KL divergence to guarantee its convexity. An additional drawback is represented by the fact that KL divergence has no Lipschitz-continuous gradient. To cope with this later issue, differently from what we did in the preliminary related work \cite{bib:LNCS2023}, we propose to replace the KL with its second-order Taylor expansion around $\bar{\bv}$. In detail, given $x,y>0$, one has that
$$ D_{KL}(x;y) \approx D_{KL}(y;y) + D_{KL}'(y;y)\,(x-y) + \frac{1}{2}D_{KL}''(y;y)\,(x-y)^2 = \frac{1}{2y}(x-y)^2. $$
Hence, we suggest to use
\begin{equation}\label{eq:define_distance_vbar}
    \mathcal{D}(\bv;\bar{\bv}) = \sum_{p} \frac{\gamma_p}{2} (v_{p} - \bar{v}_{p})^2 = \frac{1}{2}\left\Vert \bv-\bar{\bv}\right\Vert^2_\Gamma,    
\end{equation}
where $\Gamma = \mathrm{diag}(\gamma_1,\ldots,\gamma_P)$ and $$\gamma_p = \mathrm{mid}\left\{\varepsilon,\frac{1}{\bar{v}_{p}},\frac{1}{\varepsilon}\right\}, \quad (p=1,\ldots,P)$$
for a given $\varepsilon\in(0,1)$.
In simple words, the derived penalty $\mathcal{D}$ corresponds to a weighted Euclidean distance, which, similarly to the KL divergence, forces $|v_{p}|$ to be small whenever $|\bar{v}_{p}|$ is small, i.e., forces $\bv$ to have similar support to $\bar{\bv}$.

Noting that in the objective function of \eqref{eq:MICO_potential} the quantity in the parentheses equals $v_p$ and the first two constraints
can be equivalently expressed as $\bu_p \in \Delta\;(\forall p)$, where $\Delta$ represents the standard simplex in $\Re^3$, equation \eqref{eq:MIC2T} can be rewritten as 
\begin{equation}\label{eq:MICCT-mod}
    \begin{array}{rcl}
         \min\limits_{\bu,\bc,\bw,\bv}   &\;& \displaystyle \chi_\Delta(\bu) + \frac{1}{2}\|\bv\|^2 + \frac{\mu}{2}\left\Vert \bv-\bar{\bv}\right\Vert^2_\Gamma \\
         \mathrm{s.t.} & & \displaystyle \bw^\top G_p \,\bc^\top \bu_p + v_p = \bar{I}_p,\quad \forall p = 1,\ldots,P,
    \end{array}
\end{equation}
where we denoted $\chi_\Delta(\bu) \equiv \sum_{p=1}^P \chi_\Delta(\bu_p)$, and $\chi_\Delta$ represent the characteristic function of the three-dimensional simplex.
Problem \eqref{eq:MICCT-mod} is a nonlinear, nonconvex optimization problem, with the nonconvexity stemming from the multiaffine constraints, which are nonlinear constraints that remain linear with respect to each individual unknown.

To solve the optimization problem \eqref{eq:MICCT-mod} one can resort to a specialized Alternate Directions Method of Multipliers (ADMM) recently introduced and analyzed in \cite{bib:ADMMmultiaffine} for the case of multiaffine constraints. To facilitate the method presentation and convergence analysis, we will reformulate it as follows
\begin{equation}\label{eq:MICCT_multiaffine}
    \begin{array}{rcl}
         \min\limits_{\bu,\bc,\bw,\bv}  &\;& \displaystyle \phi(\bu,\bc,\bw,\bv) :=  \chi_\Delta(\bu) + g(\bv)\\
         \mathrm{s.t.} & & \displaystyle \mathcal{M}(\bu,\bc,\bw) + \bv = 0,
    \end{array}
\end{equation}
where $g(\bv)= \frac{1}{2}\|\bv\|^2 + \frac{\mu}{2}\left\Vert \bv-\bar{\bv}\right\Vert^2_\Gamma$ and $\mathcal{M}(\bu,\bc,\bw)$ describes the multiaffine part of the constraints in \eqref{eq:MICCT-mod}. It is worth observing that $\chi_\Delta(\bu)$ is a proper convex and lower semi-continuous function while $g(\bv)$ is $(\sigma_g,\,L_g)-$strongly convex, i.e., it is strongly convex with strong convexity constant $\sigma_g$ and its gradient is Lipshitz-continuous with Lipschitz constant $L_g$, where
\begin{equation}\label{eq:g_sig_L_bounds}
    \sigma_g=\frac{1}{2}\left(1+\mu\min_p{\gamma_p}\right) \ge \frac{1+\mu\varepsilon}{2}, \qquad L_g = 1+\mu\max_p{\gamma_p} \le 1+\frac{\mu}{\varepsilon}.    
\end{equation}

\noindent Given $\rho>0$, the augmented Lagrangian of \eqref{eq:MICCT_multiaffine} is 
\begin{equation} \label{eq:aug_lagr_multiaffine}
    \mathcal{L}_{\rho}(\bu,\bc,\bw,\bv,\blambda) = \phi(\bu,\bc,\bw,\bv) + \left\langle\blambda,\mathcal{M}(\bu,\bc,\bw) + \bv\right\rangle + \frac{\rho}{2}\left\|\mathcal{M}(\bu,\bc,\bw) + \bv\right\|^2,
\end{equation}
where the vector $\blambda\in\Re^P$ represents the Lagrange multipliers. An ADMM scheme can then be defined as follows: given a choice for $(\bu^0,\bc^0,\bw^0,\bv^0,\blambda^0)$, for each $k$ one obtains the new iterate $(\bu^k,\bc^k,\bw^k,\bv^k,\blambda^k)$ as
\begin{align}
    \bu^{k+1} & = \displaystyle \argmin\limits_{\bu} \mathcal{L}_{\rho}(\bu,\bc^k,\bw^k,\bv^k,\blambda^k),\label{eq:admm_method_multiaffine_U}\\
    \bc^{k+1} & = \displaystyle \argmin\limits_{\bc} \mathcal{L}_{\rho}(\bu^{k+1},\bc,\bw^k,\bv^k\blambda^k),\label{eq:admm_method_multiaffine_C}\\
    \bw^{k+1} & = \displaystyle \argmin\limits_{\bw} \mathcal{L}_{\rho}(\bu^{k+1},\bc^{k+1},\bw,\bv^k,\blambda^k),\label{eq:admm_method_multiaffine_W}\\
    \bv^{k+1} & = \displaystyle \argmin\limits_{\bv} \mathcal{L}_{\rho}(\bu^{k+1},\bc^{k+1},\bw^{k+1},\bv,\blambda^k),\label{eq:admm_method_multiaffine_V}\\
    \blambda^{k+1} & = \displaystyle \blambda^k + \rho\left(\mathcal{M}(\bu^{k+1},\bc^{k+1},\bw^{k+1}) + \bv^{k+1}\right).\label{eq:admm_method_multiaffine_lambda}
\end{align}

\subsection{ADMM subproblems solution}
In this section, we present the solution of the problems  \eqref{eq:admm_method_multiaffine_U}--\eqref{eq:admm_method_multiaffine_lambda}. First, by introducing scaled Lagrange multipliers vectors $\bzeta^k = \frac{1}{\rho}\blambda^k$, one can check that the ADMM subproblems can be rewritten as follows:
\begin{align}
    \bu^{k+1} =\;& \displaystyle \argmin\limits_{\bu} \; 
\sum_{p=1}^P\left[ \chi_\Delta(\bu_p) + \frac{\rho}{2}\left(G_p^\top\bw^k \,\bu_p^\top \bc^k + v_p^k - \bar{I}_p + \zeta_p^k\right)^2\right], \label{eq:admm_method_mod_Usubp}\\
    \bc^{k+1} =\;& \displaystyle \argmin\limits_{\bc} \;\left\|\mathcal{M}(\bu^{k+1},\bc,\bw^k) + \bv^k + \bzeta^k\right\|^2,\label{eq:admm_method_mod_Csubp}\\
    \bw^{k+1} =\;& \displaystyle \argmin\limits_{\bw} \;\left\|\mathcal{M}(\bu^{k+1},\bc^{k+1},\bw) + \bv^k + \bzeta^k\right\|^2, \label{eq:admm_method_mod_Wsubp}\\
    \bv^{k+1} =\;& \displaystyle \argmin\limits_{\bv} \; g(\bv) + \frac{\rho}{2}\left\|\mathcal{M}(\bu^{k+1},\bc^{k+1},\bw^{k+1}) + \bv + \bzeta^k\right\|^2,\label{eq:admm_method_mod_Vsubp}\\
    \bzeta^{k+1} =\;& \displaystyle \bzeta^k + \mathcal{M}(\bu^{k+1},\bc^{k+1},\bw^{k+1}) + \bv^{k+1}.\label{eq:admm_method_mod_Lupdate}
\end{align}

\noindent Problem \eqref{eq:admm_method_mod_Usubp} 
can be separated into $P$ independent problems with dimension $N$. Namely, for each $p$ one can update $\bu_p$ by solving
$$ \min\limits_{\bu_p\in\Delta} \; \left(\bm_p^\top \bu_p + l_p\right)^2,$$
where $\bm_p=G_p^\top\bw^k  \,\bc^k$ and $l_p=v^k_p -\bar{I}_p+\zeta_p^k$. Alternatively, the $P$ problems can be solved simultaneously as a unified quadratic programming problem of size $N\times P$ where the $P$ unitary simplex constraints can be replaced by nonnegativity constraints and linear equality constraints. In the experiments, we find a solution using the Interior Point - Proximal Method of Multipliers (IP-PMM) \cite{bib:DeSimone2022sirev,bib:Pougkakiotis2021}.

Due to the multiaffine nature of $\mathcal{M}$, the update of $\bc$ and $\bw$ can be easily interpreted as solving small-scale least squares problems (typical dimensions are $N=3$ and $M =10\sim 20$). More in detail, \eqref{eq:admm_method_mod_Csubp} can be rewritten as
$$\bc^{k+1} = \displaystyle \argmin\limits_{\bc} \;\left\|A^k\bc + \bv^k - \bar{I} + \bzeta^k\right\|^2,$$
where $A^k\in\Re^{P\times N}$ has rows $A^k_p = G_p^\top\bw^k\, (\bu_p^{k+1})^\top $. Similarly, subproblem \eqref{eq:admm_method_mod_Wsubp} can be reformulated as
$$\bw^{k+1} = \displaystyle \argmin\limits_{\bw} \;\;\left\|B^k\bw + \bv^k - \bar{I} + \bzeta^k\right\|^2,$$
where $B^k\in\Re^{P\times M}$ has rows $B^k_p = (\bc^{k+1})^\top \bu_p^{k+1}\, G_p^\top $.
Exact solutions of the normal equation systems can be computed cheaply to perform the two updates.

A closed-form update can be devised for $\bv$ in \eqref{eq:admm_method_mod_Vsubp} by noting that the problem can be split into $P$ 1-dimensional quadratic optimization problems. In detail, each $v_p$ is the solution to
$$ \min_v \; \frac{v^2}{2} + \frac{\mu\gamma_p}{2}\left(v-\bar{v}_p\right)^2 + \frac{\rho}{2}\left(v+z_p\right)^2, $$
where $z_p = G_p^\top\bw^{k+1} \,(\bu^{k+1}_p)^\top \bc^{k+1} - \bar{I}_p + \zeta_p^k$. This implies that at each step, one can perform an exact update of $\bv$ by setting
$$ v_p = \displaystyle\frac{\mu\gamma_p\bar{v}_p-\rho z_p}{1+\mu\gamma_p+\rho}, \qquad p=1,\ldots,P. $$

\section{Convergence of the multi-step multiaffine ADMM scheme}\label{sec:proof}
Before analyzing the convergence properties of the algorithm defined by \eqref{eq:admm_method_multiaffine_U}--\eqref{eq:admm_method_multiaffine_lambda}, it is worth mentioning that, despite their successful application in various application fields, multi-block ADMM schemes (i.e., with 3 or more blocks) are not guaranteed to converge even in the case of linear constraints and convex objective (see, e.g., \cite{bib:Chen2016}).
In the following, we will adapt some of the results from \cite{bib:ADMMmultiaffine} to the specific case of the proposed ADMM scheme, which will help gain more insights into its theoretical properties. It is worth mentioning that the results in \cite{bib:ADMMmultiaffine} focus on an extremely generic family of problems. Since we are here focusing on a very specific mathematical model, the results and/or their proofs can be simplified. For this reason, for the sake of clarity and self-containment of this work, we will explicitly report the proofs whenever there is a difference from the results in \cite{bib:ADMMmultiaffine}.
Before proceeding, it is important to understand the limits of what can be proved.

\begin{remark}
    The results we are going to prove guarantee the convergence to a point satisfying certain stationarity conditions. However, the limiting point is not guaranteed to be the solution to \eqref{eq:MICCT_multiaffine}. As far as we know, this is the best that can be done in the case of multiaffine constraints. The theory in \cite{bib:ADMMmultiaffine}, for example, only covers convergence to a solution in the case some strengthened convexity conditions hold and/or all the primal variables are involved in at least one linear constraint.
\end{remark}

\bigskip
To improve the readability of the results to follow, we will use $\bx^k$ to denote the triple $(\bu^k,\bc^k,\bw^k)$ for all $k$; i.e., $\mathcal{M}(\bx^k)$ will be used in place of $\mathcal{M}(\bu^k,\bc^k,\bw^k)$.

\bigskip
\begin{lemma}\label{lem:auglagr_variations}
    Consider the ADMM scheme described in \eqref{eq:admm_method_multiaffine_U}--\eqref{eq:admm_method_multiaffine_lambda}.
    The update of $\bv$ in \eqref{eq:admm_method_multiaffine_V} decreases the augmented Lagrangian by the quantity
    \begin{equation}\label{eq:auglagr_v_decrease}
        \Delta^k_\bv\mathcal{L} = \mathcal{L}_{\rho}(\bx^{k+1},\bv^k,\blambda^k) - \mathcal{L}_{\rho}(\bx^{k+1},\bv^{k+1},\blambda^k) = \frac{1}{2}\left\Vert\bv^{k+1}-\bv^{k}\right\Vert^2_S,
    \end{equation}
    with $S=(1+\rho)I+\mu\Gamma$. 
    The update of the Lagrangian multipliers in \eqref{eq:admm_method_multiaffine_lambda} increases the augmented Lagrangian by the quantity
    \begin{equation}\label{eq:auglagr_lambda_increase}
        \Delta^k_{\blambda}\mathcal{L} = \mathcal{L}_{\rho}(\bx^{k+1},\bv^{k+1},\blambda^{k+1}) - \mathcal{L}_{\rho}(\bx^{k+1},\bv^{k+1},\blambda^k) = \frac{1}{\rho}\left\Vert\blambda^{k+1}-\blambda^{k}\right\Vert^2.
    \end{equation}
    Furthermore, if $\left\Vert\blambda^{k+1}-\blambda^{k}\right\Vert\rightarrow 0$ then $\nabla_{\blambda} \mathcal{L}_{\rho}(\bx^{k},\bv^k,\blambda^{k})\rightarrow 0$ and every limit point $(\bx^*,\bv^*)$ of $\left\{(\bx^k,\bv^k)\right\}_{k\in\bN}$ satisfies the constraints in \eqref{eq:MICCT_multiaffine}.
\end{lemma}
\begin{proof}
    To prove \eqref{eq:auglagr_v_decrease} we first recall Lemma~6.8 in \cite{bib:ADMMmultiaffine}, thanks to which one can prove that
    $$  \Delta^k_\bv\mathcal{L} = g(\bv^{k}) - g(\bv^{k+1}) - \nabla g(\bv^{k+1})^\top(\bv^{k}-\bv^{k+1}) + \frac{\rho}{2}\left\|\bv^{k}-\bv^{k+1}\right\|^2.$$
    By observing that $g(v)$ is quadratic, i.e., it coincides with its second-order Taylor expansion, one can write
    \begin{eqnarray*}
        g(\bv^{k}) - g(\bv^{k+1}) - \nabla g(\bv^{k+1})^\top(\bv^{k}-\bv^{k+1}) &=& \frac{1}{2}(\bv^{k}-\bv^{k+1})^\top\nabla^2 g(\bv^{k+1})(\bv^{k}-\bv^{k+1})=\\
        &=& \frac{1}{2}(\bv^{k}-\bv^{k+1})^\top \left[I+\mu\Gamma\right] (\bv^{k}-\bv^{k+1}).
    \end{eqnarray*}
    This implies
    \begin{eqnarray*}
        \Delta^k_\bv\mathcal{L} &=& \frac{1}{2}(\bv^{k}-\bv^{k+1})^\top \left[I+\mu\Gamma\right] (\bv^{k}-\bv^{k+1}) + \frac{\rho}{2}\left\|\bv^{k}-\bv^{k+1}\right\|^2\\
        &=& \frac{1}{2}(\bv^{k}-\bv^{k+1})^\top \left[(1+\rho)I+\mu\Gamma\right] (\bv^{k}-\bv^{k+1}),    
    \end{eqnarray*}
    hence, \eqref{eq:auglagr_v_decrease} is proved.

The proof of \eqref{eq:auglagr_lambda_increase} comes from
$$ \Delta^k_{\blambda}\mathcal{L} = \left\langle \blambda^{k+1} - \blambda^{k},\,\mathcal{M}(\bx^{k+1}) + \bv^{k+1}\right\rangle = \rho\left\Vert\mathcal{M}(\bx^{k+1}) + \bv^{k+1}\right\Vert^2 = \frac{1}{\rho}\left\Vert\blambda^{k+1}-\blambda^{k}\right\Vert^2,$$
where we made use of \eqref{eq:admm_method_multiaffine_lambda}.
By observing that
$$\nabla_{\blambda} \mathcal{L}_\rho(\bx^{k+1},\bv^{k+1},\blambda^{k+1}) = \mathcal{M}(\bx^{k+1}) + \bv^{k+1} = \frac{1}{\rho}\left( \blambda^{k+1}-\blambda^{k} \right), $$
it follows that if $\left\Vert\blambda^{k+1}-\blambda^{k}\right\Vert\rightarrow 0$ then $\nabla_{\blambda} \mathcal{L}_{\rho}(\bx^{k},\bv^k,\blambda^{k})\rightarrow 0$. By continuity, this implies
$$ \nabla_{\blambda} \mathcal{L}_\rho(\bx^*,\bv^*,\blambda^*) = \mathcal{M}(\bx^*) + \bv^* = 0, $$
i.e., the thesis.
\end{proof}

We now introduce a couple of lemmas which will show that the sequence $\left\{\mathcal{L}_{\rho}(\bx^{k},\bv^k,\blambda^k)\right\}$ generated by ADMM converges for sufficiently large values of $\rho$. To ease the notation, we are indicating $\mathcal{L}_{\rho}^{k}=\mathcal{L}_{\rho}(\bx^{k},\bv^k,\blambda^k)$.

\bigskip
\begin{lemma}\label{lem:decreasing_aug_lagr}
    For sufficiently large $\rho$ we have that
    $$ \mathcal{L}_{\rho}(\bx^{k+1},\bv^k,\blambda^k) - \mathcal{L}_{\rho}(\bx^{k+1},\bv^{k+1},\blambda^{k+1}) \geq 0 $$
    and the sequence $\left\{\mathcal{L}_{\rho}^{k}\right\}_k$ is monotonically decreasing. Moreover, for sufficiently small $\delta$, we may choose $\rho$ such that
    \begin{equation}\label{eq:aug_lagr_decr_low_bound}
        \mathcal{L}_{\rho}^k - \mathcal{L}_{\rho}^{k+1}\geq\delta\left\|\bv^{k+1}-\bv^{k}\right\|^2.    
    \end{equation}
\end{lemma}
\begin{proof}
    To prove the positiveness of the augmented Lagrangian variation, we first show that the following upper bound holds:
    \begin{equation}\label{eq:bound_lagr_with_v}
        \left\Vert\blambda^{k+1}-\blambda^{k}\right\Vert^2 \le \left(1+\frac{\mu}{\varepsilon}\right)^2\,\left\|\bv^{k+1}-\bv^{k}\right\|^2.
    \end{equation}
    To prove it, one first has to observe that the optimality conditions for $\bv^{k+1}$ imply that
    \begin{equation}\label{eq:gradv_eq_lambda}
        0 = \nabla_{\bv}\mathcal{L}_{\rho}(\bx^{k+1},\bv,\blambda^k)\Big\vert_{\bv=\bv^{k+1}}= \nabla g(\bv^{k+1}) + \blambda^k + \rho \left( \mathcal{M}(\bx^{k+1}) + \bv^{k+1} \right) = \nabla g(\bv^{k+1}) + \blambda^{k+1},    
    \end{equation}
    i.e., $\blambda^{k}=- \nabla g(\bv^{k})$ for all $k>0$.
    Hence, one can write
    $$ \left\Vert\blambda^{k+1}-\blambda^{k}\right\Vert^2 = \left\Vert \nabla g(\bv^{k}) - \nabla g(\bv^{k+1})\right\Vert^2 \le L_g^2 \left\|\bv^{k+1}-\bv^{k}\right\|^2 = \left(1+\frac{\mu}{\varepsilon}\right)^2 \left\|\bv^{k+1}-\bv^{k}\right\|^2,$$
    i.e., \eqref{eq:bound_lagr_with_v} is proved.
    One can therefore exploit \eqref{eq:g_sig_L_bounds}, \eqref{eq:bound_lagr_with_v}, and Lemma~\eqref{lem:auglagr_variations} to derive
    \begin{eqnarray*}
        \mathcal{L}_{\rho}(\bx^{k+1},\bv^k,\blambda^k) &-& \mathcal{L}_{\rho}(\bx^{k+1},\bv^{k+1},\blambda^{k+1}) = \Delta^k_\bv\mathcal{L}-\Delta^k_{\blambda}\mathcal{L}=\\
        &=& \frac{1}{2}\left\Vert\bv^{k+1}-\bv^{k}\right\Vert^2_S - \frac{1}{\rho}\left\Vert\blambda^{k+1}-\blambda^{k}\right\Vert^2 \ge\\
        &\ge&\frac{1+\rho+\mu\varepsilon}{4}\left\Vert\bv^{k+1}-\bv^{k}\right\Vert^2 - \left(1+\frac{\mu}{\varepsilon}\right)^2\,\left\|\bv^{k+1}-\bv^{k}\right\|^2=\\
        &=& \left[\frac{1+\rho+\mu\varepsilon}{4}-\left(1+\frac{\mu}{\varepsilon}\right)^2\right]\,\left\|\bv^{k+1}-\bv^{k}\right\|^2.
    \end{eqnarray*}
    This implies that the augmented Lagrangian has a strict decrease at each iteration provided that
    $$\frac{1+\rho+\mu\varepsilon}{4}-\left(1+\frac{\mu}{\varepsilon}\right)^2>0, $$
    or, equivalently,
    \begin{equation}\label{eq:rho_lower_bound}
        \rho>4\left(1+\frac{\mu}{\varepsilon}\right)^2-1-\mu\varepsilon.
    \end{equation}
By observing that the updates \eqref{eq:admm_method_multiaffine_U}--\eqref{eq:admm_method_multiaffine_W} cannot increase the augmented Lagrangian, we can conclude that the sequence $\left\{\mathcal{L}_{\rho}^{k}\right\}_k$ is monotonically decreasing if $\rho$ satisfies \eqref{eq:rho_lower_bound}.

To prove \eqref{eq:aug_lagr_decr_low_bound} it is sufficient to observe that
$$ \mathcal{L}_{\rho}^k - \mathcal{L}_{\rho}^{k+1}\geq \Delta^k_\bv\mathcal{L}-\Delta^k_{\blambda}\mathcal{L}\geq \left[\frac{1+\rho+\mu\varepsilon}{4}-\left(1+\frac{\mu}{\varepsilon}\right)^2\right]\,\left\|\bv^{k+1}-\bv^{k}\right\|^2\geq\delta\left\|\bv^{k+1}-\bv^{k}\right\|^2 $$
provided that
$$\rho > 4\left[\left(1+\frac{\mu}{\varepsilon}\right)^2+\delta\right]-1-\mu\varepsilon.$$

\end{proof}

\begin{lemma}\label{lem:lagr_lower_bounded}
   There exists $\rho$ sufficiently large such that the sequence $\left\{\mathcal{L}_{\rho}^{k}\right\}_k$ is bounded below, and thus with Lemma~\ref{lem:decreasing_aug_lagr}, the sequence $\left\{\mathcal{L}_{\rho}^{k}\right\}_k$ is convergent.
\end{lemma}

\begin{proof}
    To prove this result, given $k>0$, consider the vector 
    $$\bz^{k+1} = \argmin \left\{g(\bz)\;\middle|\; \mathcal{M}(\bx^{k+1})+\bz=0 \right\},$$
    which satisfies
    $$ \left\Vert \bz^{k+1} - \bv^{k+1} \right\Vert = \left\Vert \mathcal{M}(\bx^{k+1}) + \bv^{k+1} \right\Vert.$$
    Subtracting $\left\langle\blambda^{k+1},\mathcal{M}(\bx^{k+1}) + \bz^{k+1}\right\rangle=0$ from $\mathcal{L}_{\rho}^{k+1}$, one has
\begin{equation}\label{eq:lam_bound_lagr_1}
    \mathcal{L}_{\rho}^{k+1} = \chi_\Delta(\bu^{k+1}) + g(\bv^{k+1}) + \left\langle\blambda^{k+1},\bv^{k+1}-\bz^{k+1}\right\rangle + \frac{\rho}{2}\left\|\mathcal{M}(\bx^{k+1}) + \bv^{k+1}\right\|^2.
\end{equation}
    We recall that the optimality conditions for $\bv^{k+1}$ imply $\blambda^{k+1} = - \nabla g(\bv^{k+1})$. By substituting this into \eqref{eq:lam_bound_lagr_1} and adding and subtracting $g(\bz^{k+1})$, one obtains

\begin{multline}\label{eq:lam_bound_lagr_2}
        \mathcal{L}_{\rho}^{k+1} = \chi_\Delta(\bu^{k+1}) + g(\bz^{k+1}) + \frac{\rho}{2}\left\|\mathcal{M}(\bx^{k+1}) + \bv^{k+1}\right\|^2 +\\
        - \left( g(\bz^{k+1}) - g(\bv^{k+1}) - \left\langle\nabla g(\bv^{k+1}),\bz^{k+1}-\bv^{k+1}\right\rangle \right).
\end{multline}
By observing that
\begin{multline*}
    g(\bz^{k+1}) - g(\bv^{k+1}) - \left\langle\nabla g(\bv^{k+1}),\bz^{k+1}-\bv^{k+1}\right\rangle = \\
    = \frac{1}{2}\left\Vert \bz^{k+1}-\bv^{k+1} \right\Vert^2_{I+\mu\Gamma}\leq  \frac{1+\frac{\mu}{\varepsilon}}{2}\left\Vert \bz^{k+1}-\bv^{k+1} \right\Vert^2 = \frac{1+\frac{\mu}{\varepsilon}}{2}\left\Vert \mathcal{M}(\bx^{k+1}) + \bv^{k+1} \right\Vert^2,
\end{multline*}
one has
\begin{equation*}
        \mathcal{L}_{\rho}^{k+1} \ge \chi_\Delta(\bu^{k+1}) + g(\bz^{k+1}) + \frac{1}{2} \left(\rho-1-\frac{\mu}{\varepsilon} \right)\left\|\mathcal{M}(\bx^{k+1}) + \bv^{k+1}\right\|^2>0, 
\end{equation*}
for $\rho>1+\frac{\mu}{\varepsilon}$. By Lemma~\ref{lem:decreasing_aug_lagr}, if
$$ \rho>\max\left\{1+\frac{\mu}{\varepsilon},\;4\left(1+\frac{\mu}{\varepsilon}\right)^2-1-\mu\varepsilon\right\} $$
the sequence $\left\{\mathcal{L}_{\rho}^{k}\right\}_k$ is monotonically decreasing and lower bounded, thus it converges.
\end{proof}

We are now ready to state the main result of this section.

\begin{theorem}\label{thm:convergence}
    Assume the sequence $\left\{(\bx^k,\bv^k,\blambda^k)\right\}_k$ produced by \eqref{eq:admm_method_multiaffine_U}--\eqref{eq:admm_method_multiaffine_lambda} is bounded, hence it has limit points. Every limit point $(\bx^*,\bv^*,\blambda^*)$ satisfies the constraint $\mathcal{M}(\bx^*) + \bv^* =0$. Furthermore, $\bv^*$ is a stationary point for the problem
    \begin{equation}\label{eq:maintheorem_v_problem}
        \begin{array}{rcl}
             \min\limits_{\bv}  &\;& \displaystyle g(\bv) = \frac{1}{2}\|\bv\|^2 + \frac{\mu}{2}\left\Vert \bv-\bar{\bv}\right\Vert^2_\Gamma\\
             \mathrm{s.t.} & & \displaystyle \mathcal{M}(\bx^*) + \bv = 0.
        \end{array}
    \end{equation}
    
\end{theorem}
\begin{proof}
    From Lemma~\ref{lem:decreasing_aug_lagr}, for sufficiently large $\rho$, we have that
    $$ \mathcal{L}_{\rho}^k - \mathcal{L}_{\rho}^{k+1}\geq\delta\left\|\bv^{k+1}-\bv^{k}\right\|^2 $$
    for a given $\delta>0$. Summing over $k$, and by Lemma~\ref{lem:lagr_lower_bounded} we have that
    $$ \mathcal{L}_{\rho}^0 \geq \mathcal{L}_{\rho}^0 - \lim_k\mathcal{L}_{\rho}^{k} \geq \delta \sum_{k=1}^{+\infty}\left\|\bv^{k+1}-\bv^{k}\right\|^2. $$
    Hence $\left\|\bv^{k+1}-\bv^{k}\right\|\rightarrow 0$, which, combined with \eqref{eq:bound_lagr_with_v}, yields $\left\Vert\blambda^{k+1}-\blambda^{k}\right\Vert\rightarrow 0$. The latter, by Lemma~\ref{lem:auglagr_variations}, implies that $\nabla_{\blambda} \mathcal{L}_{\rho}(\bx^{k},\bv^k,\blambda^{k})\rightarrow 0$ and that every limit point $(\bx^*,\bv^*)$ of $\left\{(\bx^k,\bv^k)\right\}_{k\in\bN}$ satisfies $\mathcal{M}(\bx^*) + \bv^* = 0$.

    By checking the gradient of the Lagrangian at the end of each iteration, one can verify that
    $$ \nabla_\bv \mathcal{L}_\rho(\bx^{k+1},\bv^{k+1},\blambda^{k+1}) = \nabla g(\bv^{k+1}) + \blambda^{k+1} + \rho\left(\mathcal{M}(\bx^{k+1}) + \bv^{k+1}\right)\rightarrow 0$$ 
    by \eqref{eq:gradv_eq_lambda} and by continuity of the multiaffine constraints.
    By continuity, one has that
    $$ \nabla_\bv \mathcal{L}_\rho(\bx^*,\bv^*,\blambda^*) = \nabla g(\bv^*) + \blambda^* + \rho\left(\mathcal{M}(\bx^*) + \bv^*\right) = 0,$$ 
    i.e., considering that $(\bx^*,\bv^*)$ is feasible for \eqref{eq:MICCT_multiaffine},
    $$ -\nabla g(\bv^*) = \blambda^* + \rho\left(\mathcal{M}(\bx^*) + \bv^*\right) = \blambda^*.$$ 
    Hence, $\bv^*$ satisfies the first order optimality conditions for problem \eqref{eq:maintheorem_v_problem}.
\end{proof}

\begin{remark}
    In the result above, we assumed the boundedness of the sequence generated by the ADMM scheme. This assumption is sometimes replaced by a coercivity requirement for $\phi$ over the feasible set. In our specific case, $\phi$ is coercive with respect to $\bu$ and $\bv$, but it is completely independent from $\bc$ and $\bw$. One could, of course, add some terms to $\phi$ penalising the norm of the two vectors to ensure the coercivity. However, we prefer to add the boundedness assumption and keep the model as simple as possible. The experiments we performed suggested that the norm of all the primal unknowns is non-increasing along the iterations, thus justifying the additional assumption.
\end{remark}

\subsection{Empirical convergence behaviour}

In order to analyse the convergence behaviour of the proposed ADMM scheme we analyze the history of the objective function (as done in the original paper \cite{bib:ADMMmultiaffine}) and the change in the reconstructed MRI scan, i.e.
$$ \left\Vert G^\top\bw^k \,(\bu^k)^\top \bc^k - G^\top\bw^{k-1} \,(\bu^{k-1})^\top \bc^{k-1} \right\Vert.$$

The plots in Figure~\ref{fig:convergence} below show these two metrics for 8 different instances of the problem with different levels of noise intensity and bias field.

\begin{figure}[htbp]
\centering
    \includegraphics[width=0.4\columnwidth]{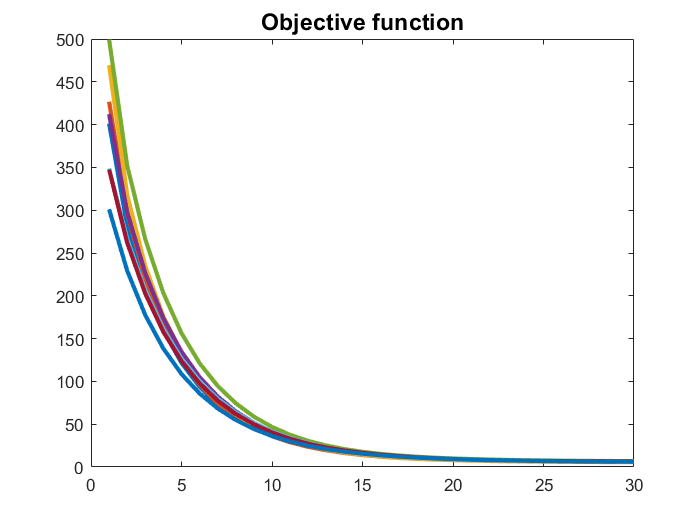}
    \includegraphics[width=0.4\columnwidth]{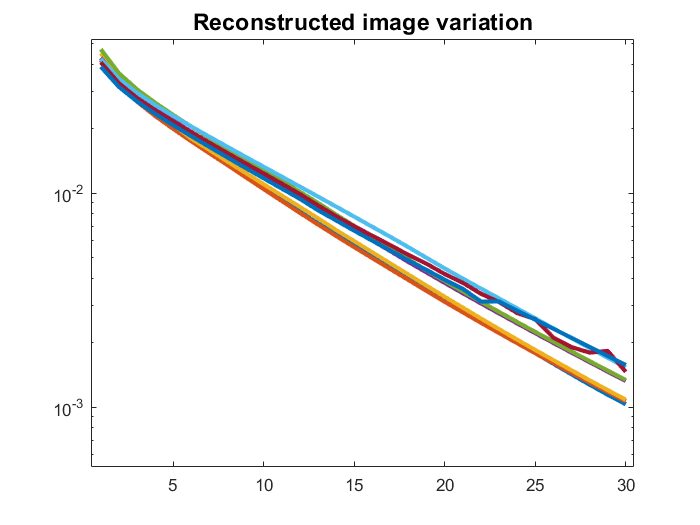}

 \caption{History of the objective function value and the change in the reconstructed MRI scan for 8 different instances of the multiaffine-constrained problem.\label{fig:convergence}}
\end{figure}

As one can see, the objective function decreases in all the cases and the reconstructed image variation falls below $2\cdot10^{-3}$ after around 30 ADMM iterations for all the instances.

\section{Numerical experiments}\label{sec:test}
All experiments are performed on the publicly available da\-ta\-set BrainWeb: Simulated Brain Database (SBD) \cite{bib:MRIsimulator}. The SBD simulator can generate MRI data of regular brain models, producing stacks of $181$ slices of $181\times127$ pixel sizes and the corresponding ground truth for cerebral tissue segmentation. Noise and bias fields can be added by the simulator to the slice to provide realistic distortions in MRI data. Our experiments ran on all stacks of T1-weighted MRI slices using several percentages of noise (\texttt{np}) and bias levels (\texttt{bl}) to analyze the reliability of the proposed algorithm in segmenting the three ROIs under investigation. For the sake of brevity, the main results performed on four slices are discussed. The selected brain MRI slices were the number 60, 64, 91, and 100\footnote{We named the four selected \textsc{BrainWeb} MRIs as \texttt{sliceXXX} where \texttt{XXX} is the number of the slice with three digits} shown in Table \ref{tab:ROIpixels}. Each slice appears with the corresponding ground truth and a list of the sizes of each single tissue expressed in the number of pixels and their percentage over the total.
\begin{table}[ht!]
    \centering \ra{1}  
    \begin{tabular}{@{}lrrrr@{}} 
  & \footnotesize \texttt{slice060} & 
  \footnotesize \texttt{slice064} & 
  \footnotesize \texttt{slice091} & 
  \footnotesize \texttt{slice100} \\
  \footnotesize  \textsc{BrainWeb} MRIs & \includegraphics[width=.12\columnwidth]{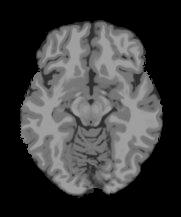} &
\includegraphics[width=.12\columnwidth]{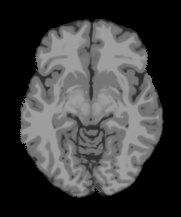} &
\includegraphics[width=.12\columnwidth]{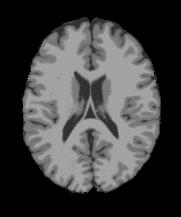} &
\includegraphics[width=.12\columnwidth]{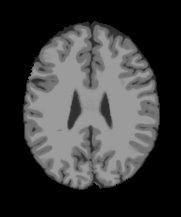}\\
\midrule
\footnotesize  \texttt{Ground Truths} & \includegraphics[width=.12\columnwidth]{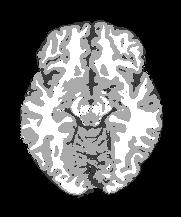} &
\includegraphics[width=.12\columnwidth]{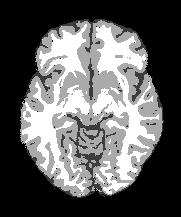} &
\includegraphics[width=.12\columnwidth]{gt_brainweb_091.jpg} &
\includegraphics[width=.12\columnwidth]{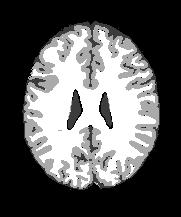}\\
 \midrule
 \footnotesize{ROI} & \multicolumn{4}{c}{\footnotesize \texttt{Number of Pixels and Percentage of the Total}} \\ 
 \midrule
 \textcolor{csf}{CSF} & $2424\;(12.6\%)$ & $2288\;(11.6\%)$ & $2849\;(14.8\%)$ & $2175\;\;\;\,(12\%)$ \\
 \textcolor{gm}{GM} & $10702\;(55.5\%)$ & $10133\;(51.2\%)$ & $6942\;(36.2\%)$ & $6572\;(36.3\%)$ \\
 \textcolor{wm}{WM} & $6152\;(31.9\%)$ & $7388\;(37.3\%)$ & $9401\;\;\;\,(49\%)$ & $9341\;(51.6\%)$ \\
\bottomrule
\end{tabular}\vspace{1mm}
 \caption{Selected MRI slices of $181 \times 127$ pixels from \textsc{BrainWeb} enhance different sizes of the three tissues,  cerebrospinal fluid (CSF), grey matter (GM), and white matter (WM), measured in number of pixels and percentage of the total.}\label{tab:ROIpixels}
\end{table}
\subsection{The SegMIC2T algorithm for MRI segmentation}
The algorithm SegMIC2T exploits the ADMM method introduced in the previous sections to solve our model \eqref{eq:MIC2T}, which acts as a correction step of the noisy and biased MRI slice and combines with it a pre- and a post-processing step, suitably. A scheme of the resulting algorithm is depicted in Algorithm~\ref{alg:SegMIC2T}. We note that any strategy to estimate noise location can be used in the initial step, and any segmentation method can be used in the last step.

\begin{algorithm}
\caption{\textsf{SegMIC2T}}
\label{alg:SegMIC2T} 
\begin{algorithmic}[1]
\State \textsf{\textbf{Noise estimation step}} 
\Statex Build a decomposition \eqref{eq:ctd} of the initial corrupted image $I$
\State \textsf{\textbf{Correction step}}
\Statex Run the ADMM scheme \eqref{eq:admm_method_mod_Usubp}--\eqref{eq:admm_method_mod_Lupdate} to determine the denoised and debiased image $J_p=(\bc^\top \bu_p)_p$
\State \textsf{\textbf{Segmentation step}}
\Statex Run the K-means clustering method to segment $J$ into the three ROIs
\end{algorithmic}
\end{algorithm}

\begin{figure}[ht!]
\centering
\begin{tabular}{ccc|c|c}
\toprule
  & \multicolumn{4}{c}{\textsf{SegMIC2T}}\\
 \cmidrule{2-5}
\footnotesize \texttt{slice060}  & \multicolumn{2}{|c|}{\scriptsize \textsf{\textbf{Noise estimation step}}} &   \multicolumn{1}{c|}{\scriptsize \textsf{\textbf{Correction step}}} & \scriptsize \textsf{\textbf{Segmentation step}}\\ 
{\scriptsize \texttt{np:} $5\%$, \texttt{bl:} $0$} & \multicolumn{1}{|c}{ \scriptsize $\bar{I}$} &  \scriptsize $\bar{\bv}$ & \footnotesize $J = (\bc^\top\bu_p)_p$ & \\
\includegraphics[width=.13\columnwidth]{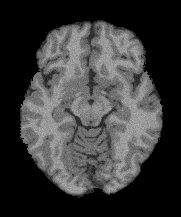} &
\multicolumn{1}{|c}{\includegraphics[width=.13\columnwidth]{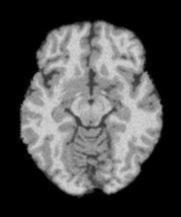}} &
\includegraphics[width=.13\columnwidth]{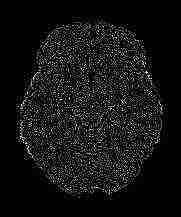} &
\includegraphics[width=.13\columnwidth]{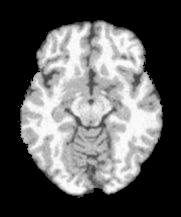} &
\includegraphics[width=.13\columnwidth]{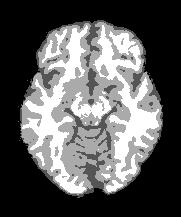}\\
 {\scriptsize \texttt{np:} $9\%$, \texttt{bl:} $40$} & \multicolumn{1}{|c}{ }& & & \\
 \includegraphics[width=.13\columnwidth]{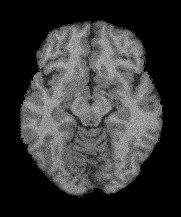} &
\multicolumn{1}{|c}{\includegraphics[width=.13\columnwidth]{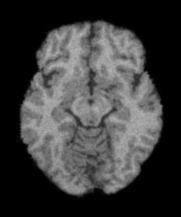}} &
\includegraphics[width=.13\columnwidth]{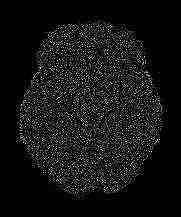} &
\includegraphics[width=.13\columnwidth]{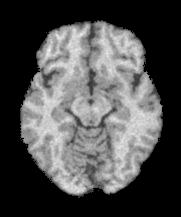} &
\includegraphics[width=.13\columnwidth]{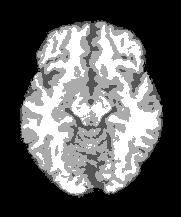}\\
\bottomrule
\end{tabular}
\caption{Workflow of the  \textsf{SegMIC2T} algorithm performed on \texttt{slice060} with low (\textit{top row}) and moderate (\textit{bottom row}) level of the combined artifacts. We indicate the noise percentage with \texttt{np} and with \texttt{bl} the bias level, respectively.
}\label{fig:SegMIC2T}
\end{figure}

We developed a \textsc{Matlab} implementation of SegMIC2T, performing image operations by built-in functions of the Image Processing toolbox. The decomposition of the input image was carried out by a non-linear low-pass and high-pass filter pair \cite{bib:ctd}, while the segmentation was realized by the K-means algorithm (setting the number of clusters to 3). {We note that the CT decomposition produced by \eqref{eq:ctd} is not unique, but a coarse decomposition is sufficient to initialize our model, since variations in the decomposition parameters have limited impact on the final segmentation accuracy \cite{bib:ctetris}}. The input parameters of the ADMM algorithm were set as follows (see \cite{bib:ctetris} for further details): $N=3$ is the number of tissues to be segmented,  the three constants $c_i, \, i=1, \ldots, N$ were initialized to $0.33$, $0.66$, and $0.99$ since the pixel intensities are normalized in $[0,1]$,  then the functions $\bu_p\in\Re^3$  were set as
$$
\bu_p=\left\lbrace \begin{array}{ll}
(1,0,0), & \mbox{ if } \bar{I}_p\in[0,0.33], \\
(0,1,0), & \mbox{ if } \bar{I}_p\in(0.33,0.66], \\
(0,0,1), & \mbox{ if } \bar{I}_p\in(0.66,1]. \\
\end{array} \right.
$$

\noindent Moreover, $M=10$ is the dimension of the third-order Legendre polynomial base $\{\bg^1, \ldots, \bg^M\}$ where $\bg^j \in\Re^P$ and $P$ is the total number of image pixels. The initialization of the $M$-vector $\mathbf{w}$ satisfies $\mathbf{w}^\top G_p=1$ for each pixel $p$. Finally, $\mathbf{v}$ starts as a vector with all $0$. The maximum number of ADMM iterations was empirically set to $30$, while the weighting KL term parameter was set as $\mu=10^{-2}$, and the correction parameter $\varepsilon=10^{-13}$, for all the test problems, and the augmented Lagrange penalty parameter was set to $\rho=10$. 

Figure~\ref{fig:SegMIC2T} illustrates the workflow of the algorithm SegMIC2T using the \texttt{slice060} with low (\textit{top row}) and moderate (\textit{bottom row}) levels of the combined artifacts.
The noise estimation step splits the original MRI slice into a cartoon image (\textit{second column}) and a texture image (\textit{third column}) \eqref{eq:ctd}. The correction step, coinciding with the application of the proposed multiaffine ADMM method, produces a denoised and debiased version of the original slice, whose $p-$th pixel corresponds to $\bc^\top \bu_p$ (\textit{fourth column}). In the last step, the segmentation of the original slice is produced by clustering the pixels in the corrected image (\textit{fifth column}). 

{We compared SegMIC2T with the MICO implementation available from \cite{MICOcode} and with TVMICO, which is an extension of the original MICO model proposed by the same authors that incorporates a Total Variation (TV) regularization term  \cite{bib:TV} in the objective function to foster the membership vector ($\bu$) to be piecewise constant. It is worth noting that in \cite{bib:Wali2023}, TVMICO was shown to outperform state-of-the-art software for MRI segmentation such as FSL\footnote{\url{https://fsl.fmrib.ox.ac.uk/fsl/fslwiki/}}, SPM\footnote{\url{http://www.fil.ion.ucl.ac.uk/spm/software/}},
and FANTASM\footnote{\url{http://mipav.cit.nih.gov/}} on similar tasks. This indicates that TVMICO is a good benchmark for the performance of SegMIC2T.
Since no code was available online, starting from the original MICO code, we built our implementation of the ADMM method for the minimization of the TVMICO model presented in~\cite{bib:Wali2023}. In terms of a computational point of view, the main differences between MICO and TVMICO are the projection of the points on the 3D standard simplex and the solution of three potentially large linear systems involving the first-order finite-difference operators characterizing the TV regularization term. In our implementation, the projection is performed, as suggested by the authors, by the algorithm described in \cite{chen2011projection}. As regards the linear systems, since the boundaries of MRI slices are usually black background pixels, periodic boundary conditions can be used for the MRI slices and the difference operators, resulting in the difference operators being Block Circulant with Circulant Blocks (BCCB). Therefore, at each step of TVMICO, the linear systems can be solved efficiently by diagonalizing the BCCB matrices using two-dimensional Discrete Fourier Transforms (DFTs). In terms of parameters' choice, we fixed the ADMM penalty parameter to $\gamma = 5$ and the TV regularization parameter to $\lambda=10^{-3}$ for all the settings (we refer the reader to \cite{bib:Wali2023} for an explanation of the role of the two parameters). It is worth mentioning that choosing the same parameters' setting for all the experiments is not limiting for a twofold reason: first, also in \cite{bib:Wali2023} the authors fix a single parameter for all the experiments; second, this makes the comparison with SegMIC2T fairer, since also in our case there is a fixed parameters' choice for all the experiments disregarding of the noise and bias level. Since the two ADMM methods (the one introduced here for SegMIC2T and the one for TVMICO) have similar per-iteration costs, we stopped the execution after 30 ADMM iterations for TVMICO.}

\subsection{Test description}
Our first test is aimed at showing the robustness of the proposed algorithm with respect to both noise and intensity bias, using three noise percentages (i.e., $5\%$, $7\%$, $9\%$), each one combined with a different bias level (i.e., $0$, $20$, $40$) on \texttt{slice091}. Figure \ref{fig:slices91} shows a comparison of the qualitative results between SegMIC2T and MICO in segmenting \texttt{slice091} corrupted by combined artifacts. The three tissues were partitioned by our method with better accuracy than MICO, without any artifacts, even with increasing noise percentages and bias levels. One can indeed observe that, for high noise levels, MICO produces inaccurate segmentations; the larger the noise, the more corruption there is in the output segmentation. 
\begin{figure}[ht!]
\centering
\begin{tabular}{cccc}
\multicolumn{4}{c}{\footnotesize \texttt{slice091}}\\ 
\midrule
\footnotesize \texttt{ground truth} &   \multicolumn{3}{c}{\footnotesize \texttt{binary masks}}  \\
\cmidrule{2-4}
\includegraphics[width=.12\columnwidth]{gt_brainweb_091.jpg} & \includegraphics[width=.12\columnwidth]{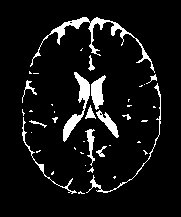} &  \includegraphics[width=.12\columnwidth]{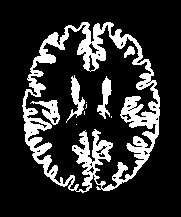} & \includegraphics[width=.12\columnwidth]{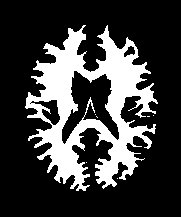} \\
& \footnotesize \texttt{\textcolor{csf}{CSF}} & \footnotesize \texttt{\textcolor{gm}{GM}} & \footnotesize \texttt{\textcolor{wm}{WM}} \\
\midrule
 & \footnotesize \texttt{np} $5\%$,  \texttt{bl} $0$ &  \footnotesize  \texttt{np} $7\%$, \footnotesize \texttt{bl} $20$ & \footnotesize \texttt{np} $9\%$, \texttt{bl} $40$\\
 &
\includegraphics[width=.12\columnwidth]{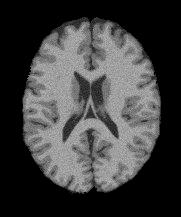} &
\includegraphics[width=.12\columnwidth]{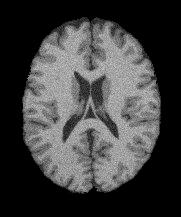} &
\includegraphics[width=.12\columnwidth]{brainweb_pn9_rf40_091.jpg}\\

& \multicolumn{3}{c}{\footnotesize \texttt{segmentations}}\\ \cmidrule{2-4}
{\footnotesize \textsf{MICO}} &\includegraphics[width=.12\columnwidth]{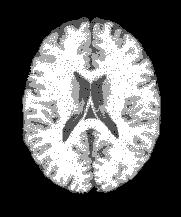} &
\includegraphics[width=.12\columnwidth]{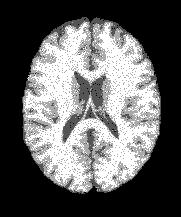} &
\includegraphics[width=.12\columnwidth]{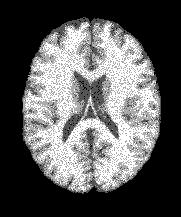}\\ \midrule
%
%

%
 \footnotesize \textsf{\textsf{SegMIC2T}} & \includegraphics[width=.12\columnwidth]{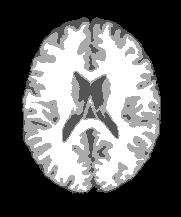} &
\includegraphics[width=.12\columnwidth]{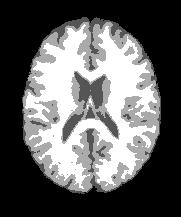} &
\includegraphics[width=.12\columnwidth]{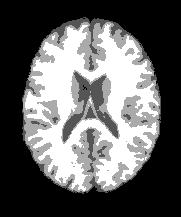} \\\midrule
\end{tabular}
\caption{The first row shows the ground-truth of the \textsc{BrainWeb} slice n. 91 and the binary masks of the three tissues, CSF, GM, and WM. The second row exhibits the corresponding MRI slice corrupted by the following three combinations of the noise percentage, \texttt{np}, and bias level, \texttt{bl}, (from left to right): $5\%$ noise and $0-$level bias, $7\%$ noise and $20-$level bias, and $9\%$ noise and $40-$level bias. The third and fourth
rows present the segmentation results produced by MICO, 
and SegMIC2T, respectively.}\label{fig:slices91}
\end{figure}
Applying the two segmentation methods to other brain slices and different noise and bias level combinations produced similar qualitative results. Usually, when the noise isn't negligible, SegMIC2T outperforms MICO in the accuracy of the qualitative results. 
To compare the performance of the two segmentation methods quantitatively, some metrics among the most adopted in the literature were evaluated on the single tissue of the brain MRI slices. All the metrics, whose definition is specified in Table~\ref{tab:segmetrics}, assume values in $[0,1]$, and higher values indicate better-performed segmentation. Each metric refers to a binary segmentation mask for the specific ROI, i.e., one per CSF, one per WM, and one per GM, respectively (see the first row of Figure \ref{fig:slices91} for an example of the three binary masks on \texttt{slice091}).
Following the standard conventions, we adopted the following notations: 
\begin{itemize}
    \item $TP$, True Positive, is the number of pixels correctly recognized as part of the ROI;
    \item $TN$, True Negative, is the number of pixels correctly identified as not part of the ROI;
    \item $FP$, False Positive, is  the number of pixels being misclassified as belonging to the ROI;
    \item $FN$, False Negative, is the number of pixels in the ROI being misclassified as part of other tissues.
\end{itemize}
From this, the binary segmentation ($S$) of the single tissue and the corresponding binary mask ($BM$) can be obtained as $S=TP+FP$ and $BM=TP+FN$, respectively.

\begin{table}[ht]
    \centering
    \begin{tabular}{lccc}
    \toprule
         \footnotesize\textbf{metrics} & \footnotesize\textbf{otherwise known as} & \footnotesize\textbf{equation} & \footnotesize\textbf{description} \\
    \midrule     \footnotesize\texttt{Jaccard} & \footnotesize Intersection over  & \footnotesize  \multirow{2}{*}{$\frac{TP}{TP+FP+FN}$} &  \footnotesize measures the similarity\\
         & \footnotesize Union score  & & \footnotesize among $BM$ and $S$\\
    \midrule
    \footnotesize\texttt{Sensitivity} & \footnotesize True Positive Rate  & \multirow{2}{*}{$\frac{TP}{TP+FN} $} & \footnotesize  measures how many $TP$ \\
     & \footnotesize Recall & & \footnotesize are detected in $S$\\
     \midrule
     \footnotesize\texttt{Specificity} & \footnotesize True Negative Rate & \footnotesize  \multirow{2}{*}{$\frac{TN}{TP+TN}$} & \footnotesize measures how many $TN$ \\
     & & & \footnotesize are detected in $S$  \\
     \midrule
     \footnotesize \texttt{Dice} & \footnotesize F measure & \footnotesize \multirow{2}{*}{$\frac{2TP} {2TP+FN+FP}$} & \footnotesize measures the overlap  \\
     & \footnotesize{F1 score} & & \footnotesize between $BM$ and $S$ \\
    \bottomrule \\
    \end{tabular} 
    \caption{Metrics adopted to evaluate the performance of the segmentation methods.}
    \label{tab:segmetrics}
\end{table}
The performance of MICO and SegMIC2T in partitioning the three ROIs of \texttt{slice091} were detailed in Table \ref{tab:metrics60}. Summarizing, SegMIC2T exhibited the highest values of all quantitative metrics, confirming its good performance even when the percentage of noise and the bias level were increased.

\begin{figure}[!ht]
\centering
\begin{tabular}{ccccc}
\multicolumn{5}{c}{\footnotesize \texttt{slice064}} \\
\midrule
{\footnotesize \texttt{np:} $5\%$, \texttt{bl:} $0$} & \multicolumn{3}{c}{\footnotesize \texttt{segmentations}} \\
\includegraphics[width=.13\columnwidth]{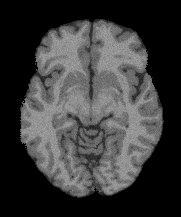} &
\includegraphics[width=.13\columnwidth]{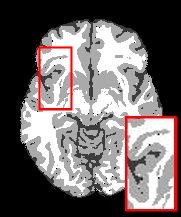} &
\includegraphics[width=.13\columnwidth]{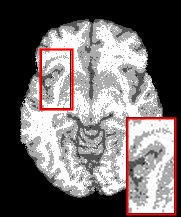} &
\includegraphics[width=.13\columnwidth]{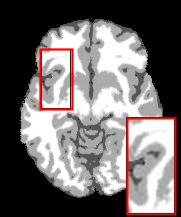}
& \includegraphics[width=.13\columnwidth]{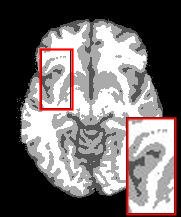}\\
\footnotesize \texttt{np:} $9\%$, \texttt{bl:} $40$ \\
\includegraphics[width=.13\columnwidth]{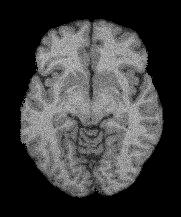} & 
\includegraphics[width=.13\columnwidth]{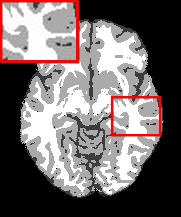} &
\includegraphics[width=.13\columnwidth]{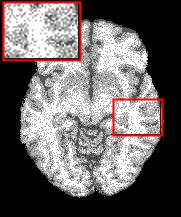} &
\includegraphics[width=.13\columnwidth]{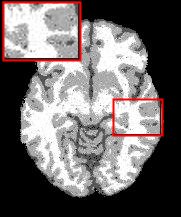} & 
{\includegraphics[width=.13\columnwidth]{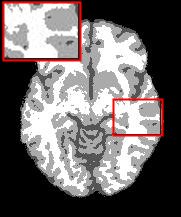}}\\\vspace*{-2mm}
 & \footnotesize \texttt{ground truth} & \footnotesize \textsf{MICO} & \footnotesize \textsf{TVMICO} & \footnotesize \textsf{SegMIC2T} \\[2mm]
\midrule 

\end{tabular} 
\caption{Comparison of \textsf{MICO}, \textsf{TVMICO} and \textsf{SegMIC2CT} on the segmentation of \texttt{slice064} with slight and moderate combined artifacts. \texttt{np}: noise percentage, \texttt{bl}: bias level.}\label{fig:seg064}
\end{figure}
To point out the denoising effect of SegMIC2T based on our model, we produced a visual comparison against MICO and TVMICO on \texttt{slice064}, reported in Figure \ref{fig:seg064}. The figure shows results for two different levels of artifacts: a low level, with $5\%$ noise and no bias, and a moderate level, with $9\%$ noise and $40-$level bias. TVMICO seems comparable to SegMIC2T in segmenting noisy slices, but its effectiveness seems to decrease when the noise increases, as we can note in Figure \ref{fig:seg064} in the highlighted zooms. 
\setlength{\tabcolsep}{8pt}
\begin{table}[ht]
\begin{footnotesize}
    \centering \ra{1}  
    \begin{tabular}{@{}lccccr@{}} 
\toprule
\textsc{\textbf{Method}} & \multicolumn{4}{c}{\textsc{\textbf{Segmentation Metrics}}} & \textsc{\textbf{ROI}}\\
\cmidrule{2-5}
    & \footnotesize \texttt{Jaccard} & \footnotesize \texttt{Sensitivity} & \footnotesize \texttt{Specificity} & \footnotesize \texttt{Dice} &  \footnotesize \texttt{slice091}   \\
\toprule
 & \multicolumn{4}{c}{\footnotesize \texttt{np:} $5\%$, \texttt{bl:} $0$} &  \\
\hline

\multirow{3}{*}{MICO} & 0.831 & 0.913 & 0.992 & 0.908 & \textcolor{csf}{CSF}\\ 
                    & 0.721 & 0.833 & 0.965 & 0.838 &\textcolor{gm}{GM}\\
                    & 0.841 & 0.906 & 0.975 & 0.913 & \textcolor{wm}{WM}\\
\hline
\multirow{3}{*}{TVMICO} & 0.848 & 0.914 & 0.994 & 0.918& \textcolor{csf}{CSF}\\ 
                    & 0.846 & 0.923 & 0.980 & 0.917& \textcolor{gm}{GM}\\
                    & 0.926 & 0.965 & 0.987 & 0.962& \textcolor{wm}{WM}\\
\hline
\multirow{3}{*}{SegMIC2T} & 0.862 & 0.924 & 0.994 & 0.926& \textcolor{csf}{CSF}\\ 
                    & 0.844 & 0.918 & 0.980 & 0.916 &\textcolor{gm}{GM}\\
                    & 0.939 & 0.960 & 0.993 & 0.969 & \textcolor{wm}{WM}\\
\hline
& \multicolumn{4}{c}{\footnotesize \texttt{np:} $7\%$, \texttt{bl:} $20$} &  \\

\hline
\multirow{3}{*}{MICO} & 0.753 & 0.875 & 0.987 & 0.859 & \textcolor{csf}{CSF}\\ 
                    & 0.606 & 0.753 & 0.946 & 0.755& \textcolor{gm}{GM}\\
                    & 0.756 & 0.847 & 0.962 & 0.861& \textcolor{wm}{WM}\\
                    \hline
\multirow{3}{*}{TVMICO} & 0.825 & 0.902 & 0.992 & 0.904 & \textcolor{csf}{CSF}\\ 
                    & 0.828 & 0.912 & 0.978 & 0.906 & \textcolor{gm}{GM}\\
                    & 0.914 & 0.960 & 0.984 & 0.955 & \textcolor{wm}{WM}\\
\hline
\multirow{3}{*}{SegMIC2T} & 0.844 & 0.924 & 0.992 & 0.916 & \textcolor{csf}{CSF}\\ 
                    & 0.803 & 0.907 & 0.971 & 0.891 &\textcolor{gm}{GM}\\
                    & 0.904 & 0.926 & 0.992 & 0.949 & \textcolor{wm}{WM}\\ 
\hline
& \multicolumn{4}{c}{\texttt{np:} $9\%$ - \texttt{bl:} $40$} &  \\
\hline
\multirow{3}{*}{MICO} & 0.642 & 0.822 & 0.977 & 0.782 & \textcolor{csf}{CSF}\\ 
                    & 0.511 & 0.691 & 0.922 & 0.677 &\textcolor{gm}{GM}\\
                    & 0.675 & 0.766 & 0.957 & 0.806 & \textcolor{wm}{WM}\\
\hline
\multirow{3}{*}{TVMICO} & 0.758 & 0.862 & 0.989 & 0.862& \textcolor{csf}{CSF}\\ 
                    & 0.801 & 0.894 & 0.974 & 0.889 &\textcolor{gm}{GM}\\
                    & 0.899 & 0.953 & 0.981 & 0.947 & \textcolor{wm}{WM}\\
\hline 
\multirow{3}{*}{SegMIC2T} &
                    0.832 & 0.909 & 0.992 & 0.908 & \textcolor{csf}{CSF}\\
                    & 0.781 & 0.864 & 0.976 & 0.877 &\textcolor{gm}{GM}\\
                    & 0.896 & 0.946 & 0.982 & 0.945 & \textcolor{wm}{WM}\\
\bottomrule
    \end{tabular}\vspace{1mm}
    \caption{Segmentation metrics results on \texttt{slice091} corrupted by increasing noise percentages (\texttt{np}) and bias levels (\texttt{bl}).}
    \label{tab:metrics60}
\end{footnotesize}
\end{table}
For a broader comparison of the reliability and efficiency of SegMIC2T over MICO and TVMICO in segmenting the three tissues, the three algorithms were run on the four MRI brain slices in Table \ref{tab:ROIpixels}, corrupted with different combinations of noise, i.e., $5\%$, $7\%$, and $9\%$, and three bias levels, $0$, $20$, and $40$. This results in a total number of 36 test problems. We report in Figure~\ref{fig:boxplotall} boxplots for the three algorithms summarising the results we obtained in terms of the four segmentation metrics grouped by tissue type (CSF, GM, WM).   

\begin{figure}[ht!]
\centering
\begin{tabular}{c}
\hspace{-5mm}\includegraphics[width=0.49\columnwidth]{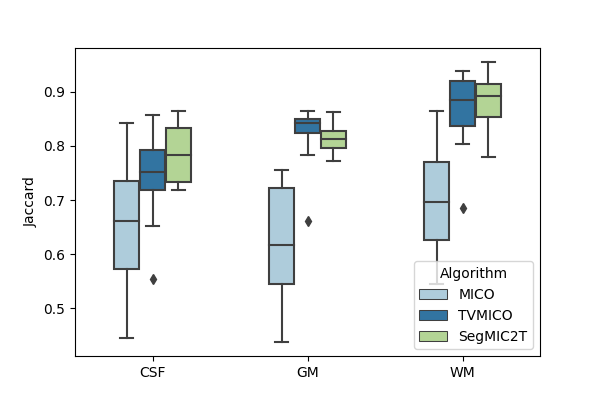} 
\includegraphics[width=0.49\columnwidth]{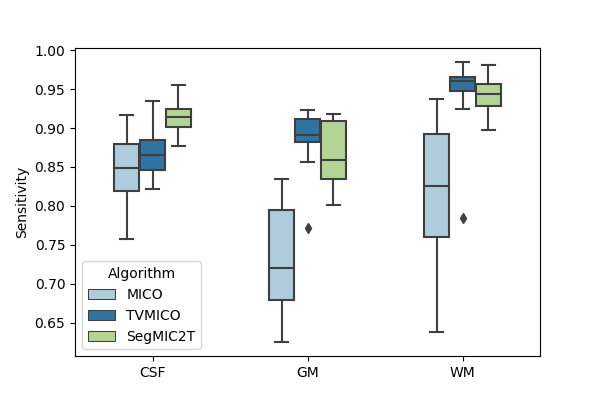} \\[2mm]
\hspace{-5mm}\includegraphics[width=0.49\columnwidth]{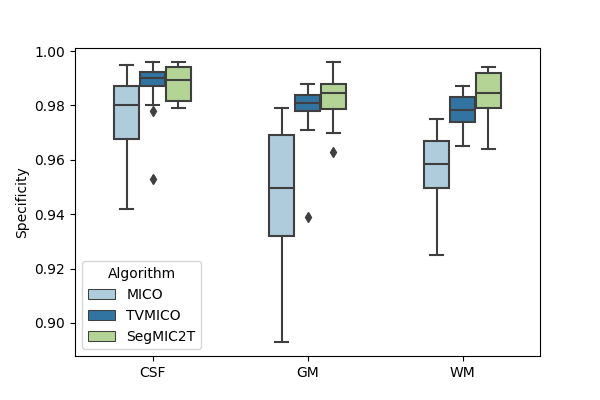}
\includegraphics[width=0.49\columnwidth]{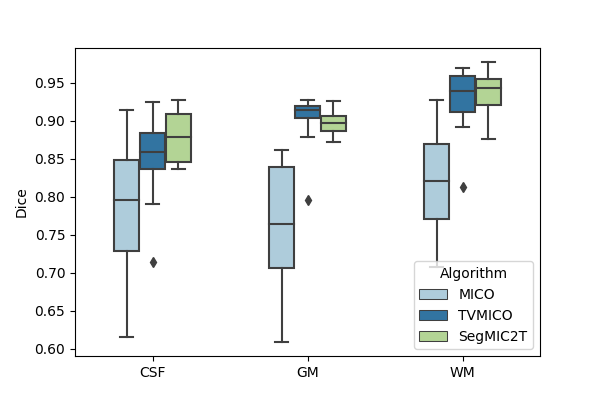}
\end{tabular}
\caption{Boxplots (grouped by ROI) for Jaccard (top left), Sensitivity  (top right), Specificity  (bottom left), and Dice  (bottom right) similarity coefficient for the segmentation of all the test problems.}\label{fig:boxplotall}
\end{figure} 

{From the boxplots, it is clear that MICO exhibits the worst metric results due to its sensitivity to noise. SegMIC2T outperforms both MICO and TVMICO in terms of \texttt{Specificity} on all three tissues and shows comparable performances with TVMICO on the other metrics. One interesting observation is that the proposed algorithm consistently outperforms the two competitors in the identification of the cerebrospinal fluid, especially in terms of specificity, i.e., true positive rate. In our opinion, this is an important feature of the proposed algorithm; indeed, despite the CSF being the smallest of the three tissues, its correct detection is essential in the diagnosis of certain diseases, such as pediatric hydrocephalus (see, e.g., \cite{Grimm2020}).
}

\section{Conclusions}\label{sec:conc}
{In summary, MRI segmentation is a critical step in analyzing the acquired slices corrupted by noise and other artifacts, especially in the context of brain imaging. Overcoming challenges such as intensity inhomogeneity, noise, and artifacts is essential for obtaining reliable results. Advancements in numerical methods and technology continue to play a significant role in improving the accuracy and efficiency of MRI segmentation for both clinical and research applications.\\
Here, we propose a new model aimed at removing noise and inhomogeneity of the image intensity. We introduced an ADMM method for the minimization of the nonlinear optimization problem with multiaffine constraints arising from the proposed model and analyzed its convergence properties. When combined with a K-means method for the segmentation of the corrected image, the proposed model seems to be effective in enhancing the segmentation of brain MRI slices affected by different noise levels and artifacts, surpassing several state-of-the-art methods for brain MRI segmentation in both qualitative and quantitative assessments. In future work, different function bases will be investigated to improve the correction step accuracy and efficiency, and a possible extension to 4D MRI segmentation could be considered.}

\section*{Acknowledgments}
The work was partially supported by
Italian Ministry of University and
Research (MIUR), PRIN Projects: \emph{Numerical Optimization with Adaptive Accuracy and Applications to Machine Learning
}, grant n. 2022N3ZNAX, \emph{A mathematical approach to inverse problems arising in cultural heritage preservation and dissemination
}, grant n. P2022PMEN2, and 
by Istituto Nazionale di Alta Matematica -
Gruppo Nazionale per il Calcolo Scientifico (INdAM-GNCS).\\[1mm] 
We would like to express our gratitude to Simona Sada (ICAR-CNR) for her technical support.

\end{document}